\documentclass[11pt]{article}
\usepackage{amsmath}
\usepackage{amssymb}
\usepackage{amsthm}

\def\clift#1{#1^{\scriptscriptstyle{\mathrm{C}}}}
\def\hlift#1{#1^{\scriptscriptstyle{\mathrm{H}}}}
\def\vlift#1{#1^{\scriptscriptstyle{\mathrm{V}}}}
\def\H#1{{#1}^{\scriptscriptstyle H}}
\def\V#1{{#1}^{\scriptscriptstyle V}}

\def\dv#1{d^{\scriptscriptstyle V}_{#1}}
\def\dh#1{d^{\scriptscriptstyle H}_{#1}}
\def\DV#1{{\rm D}^{\scriptscriptstyle V}_{#1}}
\def\DH#1{{\rm D}^{\scriptscriptstyle H}_{#1}}

\def\fpd#1#2{{\displaystyle\frac{\partial #1}{\partial #2}}}
\def\spd#1#2#3{{\displaystyle\frac{\partial^2 #1}{\partial #2\partial #3}}}
\def\lie#1{{\mathcal{L}}_{#1}}
\def\del{\nabla}
\def\dddotqi{\raisebox{-1.4pt}{${\stackrel{\raisebox{-0.7pt}{.\kern-1pt.\kern-1pt.}}{q}}^i$}}
\def\cinfty#1{C^{\scriptscriptstyle\infty}(#1)}

\def\R{\mathbb{R}}
\def\T{{\bf T}}

\def\onehalf{{\textstyle\frac12}}
\def\onethird{{\textstyle\frac13}}

\def\onefourth{{\textstyle\frac14}}
\def\threefourth{{\textstyle\frac34}}

\def\oneeights{{\textstyle\frac18}}

\def\forms#1#2{{\textstyle\bigwedge}^{#1}(#2)}
\def\Vforms#1#2{{\textstyle V}^{#1}(#2)}
\def\vectorfields#1{{\cal X}(#1)}
\def\tvectorfields{\vectorfields{\tau}}

\def\conn#1#2#3{\setbox1=\hbox{$\scriptstyle{#2}{#3}$}%
\setbox2=\hbox to\wd1{$\hfil\scriptstyle{#1}\hfil$}
\Gamma^{\!\box2}_{\!\box1}}

\def\hook{{\mathchoice
{\vrule height 0pt depth 0.4pt width 3pt \vrule height 5pt depth 0.4pt
\kern 3pt}
{\vrule height 0pt depth 0.4pt width 3pt  \vrule height 5pt depth
0.4pt\kern 3pt}
{\vrule height 0pt depth 0.2pt width 1.5pt  \vrule height 3pt depth
0.2pt width 0.2pt \kern 1pt}
{\vrule height 0pt depth 0.2pt width 1.5pt  \vrule height 3pt depth
0.2pt width 0.2pt \kern 1pt} }}

\newtheorem{thm}{Theorem}
\newtheorem{lem}{Lemma}
\newtheorem{prop}{Proposition}
\newtheorem{cor}{Corollary}

\begin{document}

\title{The inverse problem for Lagrangian systems with certain non-conservative forces}
\author{T.\ Mestdag$^a$, W.\ Sarlet$^{a,b}$ and M.\ Crampin$^a$\\[2mm]
{\small ${}^a$Department of Mathematics, Ghent University }\\
{\small Krijgslaan 281, B-9000 Ghent, Belgium}\\[1mm]
{\small ${}^b$Department of Mathematics and Statistics, La Trobe University}\\
{\small Bundoora, Victoria 3086, Australia}
}

\date{}

\maketitle

\begin{quote}
{\small {\bf Abstract.} We discuss two generalizations of the
inverse problem of the calculus of variations, one in which a given
mechanical system can be brought into the form of Lagrangian equations
with non-conservative forces of a generalized Rayleigh dissipation
type, the other leading to Lagrangian equations with so-called
gyroscopic forces. Our approach focusses primarily on obtaining
coordinate-free conditions for the existence of a suitable
non-singular multiplier matrix, which will lead to an equivalent
representation of a given system of second-order equations as one
of these Lagrangian systems with non-conservative forces.
\\[1mm]
}
\end{quote}

\section{Introduction}

The inverse problem of Lagrangian mechanics is the question: given a
system of second-order ordinary differential equations, under what
circumstances does there exist a regular Lagrangian function, such
that the corresponding Lagrange equations are equivalent (i.e.\ have
the same solutions) as the original equations. Locally, the question
can be translated immediately into more precise terms as follows:
considering a given second-order system in normal form
\begin{equation}
\ddot{q}^i = f^i(q,\dot{q}), \label{normal}
\end{equation}
which (for the time being) we take to be autonomous for simplicity,
what are the conditions for the existence of a symmetric, non-singular
multiplier matrix $g_{ij}(q,\dot{q})$ such that
\[
g_{ij}(\ddot{q}^j - f^j(q,\dot{q})) \equiv
\frac{d}{dt}\left(\fpd{L}{\dot{q}^i}\right) - \fpd{L}{q^i}
\]
for some $L$. Clearly $(g_{ij})$, if it exists, will become the
Hessian of the Lagrangian $L$. The literature on this problem is
extensive; the conditions for the existence of $L$ are usually
referred to as the Helmholtz conditions, but these can take many
different forms depending on the mathematical tools one uses and on
the feature one focusses on. For a non-exhaustive list of different
approaches see \cite{Sant}, \cite{Cra81}, \cite{Sa}, \cite{Mor},
\cite{AT92}, \cite{CSMBP}, \cite{GM1}, \cite{APST},
\cite{olgageoff}, \cite{Buca}. In this paper, the tools stem from
differential geometry and therefore provide coordinate-free results.
In addition, while we will actually study generalizations of the
above problem which allow for certain classes of non-conservative
forces, the attention will be mainly on conditions on the multiplier
$g$.

We will consider two types of non-conservative forces, leading to
Lagrangian equations of one of the following forms:
\begin{equation}
\frac{d}{dt}\left(\fpd{L}{\dot{q}^i}\right) - \fpd{L}{q^i} =
\fpd{D}{\dot{q}^i}, \label{dissip1}
\end{equation}
or
\begin{equation}
\frac{d}{dt}\left(\fpd{L}{\dot{q}^i}\right) - \fpd{L}{q^i} =
\omega_{ki}(q)\dot{q}^k, \quad \omega_{ki}= - \omega_{ik}.
\label{dissip2}
\end{equation}
In the first case, when the function $D$ is quadratic in the
velocities (and $-D$ is positive definite) the classical terminology
is that we have dissipation of Rayleigh type (see e.g.\
\cite{Goldstein}); we will not put restrictions on the form of $D$,
however.  In the second case, in which the existence of a $D$ as in
(\ref{dissip1}) is excluded, the right-hand side is often referred to
as a gyroscopic force (see e.g.\ \cite{Rosenberg}).

Perhaps we should specify first what we will not do in this paper.  In
older contributions to the inverse problem for dissipative systems,
such as \cite{napoli}, the emphasis was on trying to recast a
dissipative system into the form of genuine Euler-Lagrange equations,
that is to say that in the case of given equations of type
(\ref{dissip1}) one would try to find a different function $L'$ such
that the Euler-Lagrange equations of $L'$ are equivalent to the
given system.  In contrast, our goal here is to study under what
circumstances a given second-order system in normal form
(\ref{normal}) can be recast into the form (\ref{dissip1}) (or
(\ref{dissip2})) for some functions $L$ and $D$ (or $L$ and
$\omega_{ki}$).

In order to explain our objectives in more precise terms, let us
recall first some of the different ways of characterizing the inverse
problem conditions in the classical situation.  The natural
environment for a second-order system is a tangent bundle $TQ$, with
coordinates $(q,v)$ say, where it is represented by a vector field
$\Gamma$ of the form
\begin{equation}
\Gamma = v^i\fpd{}{q^i} + f^i(q,v)\fpd{}{v^i}. \label{Gamma}
\end{equation}
If $S=(\partial/\partial v^i)\otimes dq^i$ denotes the type $(1,1)$
tensor field which characterizes the canonical almost tangent
structure on $TQ$ \cite{Cra83, Gri1}, $\Gamma$ represents a
Lagrangian system provided there exists a regular Lagrangian
function $L$ such that (see e.g.\ \cite{SaCaCr84})
\begin{equation}
\lie{\Gamma}(S(dL))=dL; \label{thetaL1}
\end{equation}
$\theta_L:=S(dL)$ is the Poincar\'e-Cartan 1-form.  The above
condition is perhaps the most compact formulation of the problem, but
has little or no practical value when it comes to testing whether such
an $L$ exists for a given $\Gamma$.  A shift of attention towards the
existence of a multiplier leads to the following necessary and
sufficient conditions \cite{Cra81}: the existence of a non-degenerate
2-form $\omega\in\forms{2}{TQ}$, such that
\begin{equation}
\lie{\Gamma}\omega=0, \quad \omega(\V{X},\V{Y})=0, \quad
i_{\H{Z}}d\omega(\V{X},\V{Y})=0,
\quad \forall X,Y,Z\in\vectorfields{M}. \label{Mike}
\end{equation}
Here $\V{X}$ and $\H{X}$ refer to the vertical and horizontal lift
of vector fields, respectively. The latter makes use of the
canonical Ehresmann connection on $\tau:TQ\rightarrow Q$ associated
with a given second-order vector field $\Gamma$: in coordinates, the
vertical and horizontal lift are determined by
\begin{equation}
V_i:=\V{\fpd{}{q^i}} = \fpd{}{v^i}, \qquad H_i:=\H{\fpd{}{q^i}} =
\fpd{}{q^i} - \Gamma_i^j\fpd{}{v^j}, \quad \mbox{where} \quad
\Gamma_i^j= - \onehalf \fpd{f^j}{v^i}. \label{conn}
\end{equation}
Such a 2-form $\omega$ will be closed, hence locally exact, and as such
will be the exterior derivative $d\theta_L$ for some Lagrangian $L$.
At this point it is interesting to observe that the $2n\times 2n$
skew-symmetric component matrix of $\omega$ is completely determined by
the $n\times n$ symmetric matrix
\[
g_{ij} = \fpd{^2L}{v^i\partial v^j}.
\]
The matrix $(g_{ij})$ geometrically represents the components of a
$(0,2)$ symmetric tensor field $g$ along the tangent bundle
projection $\tau$, and the relationship between $\omega$ on $TQ$ and
$g$ along $\tau$ has an intrinsic meaning as well: $\omega$ is the
K\"ahler lift of $g$ (see \cite{MaCaSaII}). A more concise
formulation of the Helmholtz conditions therefore, when viewed as
conditions on the multiplier $g$, makes use of the calculus of
derivations of forms along $\tau$, as developed in \cite{MaCaSaI,
MaCaSaII}. We will show in the next section how both the conditions
(\ref{thetaL1}) and (\ref{Mike}) have an equivalent formulation in
those terms, and this will be the basis for the generalization to
Lagrangian systems with non-conservative forces, which will be the
subject of the subsequent sections.

The first authors to discuss the inverse problem, in the sense of
analyzing the conditions which a given representation of a
second-order system must satisfy to be of the form (\ref{dissip1}),
were Kielau and Maisser \cite{germ1}.  We showed in \cite{zammnote}
how the results they obtained via an entirely analytical approach can
in fact be reduced to a smaller set.  But we also argued in the
concluding remarks of that paper that the more important issue is the
one we formulated above, which starts from a normal form
representation of the dynamical system.  For that purpose it is better
to approach the problem in a coordinate-independent way, i.e.\ to make
use of the tools of differential geometry already referred to.  We
will see that the methods we will develop for the dissipative case
(\ref{dissip1}) apply equally to the gyroscopic case (\ref{dissip2}).
To the best of our knowledge the latter problem has not been dealt with
before in its entirety (though a relevant partial result has been
published in \cite{olgageoff}).  An additional advantage of the
coordinate-independence of our conditions is that they cover without
extra effort results such as those derived in \cite{germ2} for the
description of Lagrangian systems in `nonholonomic velocities'.  In
Section~3 we follow the lines of the construction of Helmholtz
conditions on the multiplier $g$ for the standard inverse problem, and
arrive in this way at necessary and sufficient conditions which involve
$g$ and $D$ in the dissipative case, and $g$ and $\omega$ in the
gyroscopic situation.  At the end of this section we briefly discuss
how the partial result mentioned above is related to our work.  In
Section~4 we succeed in eliminating the unknown $D$ and $\omega$
altogether to arrive at necessary and sufficient conditions involving
the multiplier $g$ only.  This is particularly interesting, because a
given $\Gamma$ may actually admit multiple representations of the form
(\ref{dissip1}) for example.  In other words, different choices of a
multiplier $g$ may exist, which each require an adapted (generalized)
dissipation function $D$ to match the required format.  In fact it
cannot be excluded that a given $\Gamma$ may actually have
representations in the form (\ref{dissip1}) and (\ref{dissip2}) at the
same time, of course with different multipliers $g$ (and thus different
Lagrangians $L$).  We will encounter such situations among the
illustrative examples discussed in Section~5, where we also briefly
indicate in the concluding remarks how the whole analysis can be
carried over to the case of time-dependent systems.  In an appendix we
make an excursion to a different geometrical approach which in fact is
essentially time-dependent: we use techniques from the theory
of variational sequences to relate our results more
closely, at least in the dissipative case, to those obtained in
\cite{germ1}, which after all was the work which first brought this
subject to our attention.

\section{Basic set-up}

In order to keep our analysis reasonably self-contained, we need to
recall the basics of the calculus of derivations of forms along the
tangent bundle projection $\tau:TQ\rightarrow Q$. Vector fields
along $\tau$ are sections of the pull-back bundle $\tau^*TQ
\rightarrow TQ$ and constitute a module over $\cinfty{TQ}$, denoted
by $\tvectorfields$. Likewise, a $k$-form along $\tau$ assigns to
every point  $v_q\in TQ$ an exterior $k$-form at $q=\tau(v_q)\in Q$;
we use the symbol $\forms{}{\tau}$ for the $\cinfty{TQ}$-module of
scalar forms along $\tau$ and $\Vforms{}{\tau}$ for the module of
vector-valued forms. The theory of derivations of such forms, as
established in \cite{MaCaSaI, MaCaSaII}, follows closely the
pioneering work of Fr\"olicher and Nijenhuis \cite{FrNi56}. The
difference is that there is a natural vertical exterior derivative
$\dv{}$ available, but a full classification requires an additional
horizontal exterior derivative $\dh{}$, which must come from a given
connection: in our situation, this is the connection associated with
$\Gamma$ mentioned earlier. We limit ourselves here to a brief
survey of the concepts and properties we will need. An elaborate
version of the theory (with rather different notations) can also be
found in \cite{Szila}.

Elements of $\forms{}{\tau}$ in coordinates look like forms on the
base manifold $Q$ with coefficients which are functions on $TQ$.
Thus they can be seen also as so-called semi-basic forms on $TQ$,
and we will generally make no notational distinction between the two
possible interpretations. It is clear that derivations of such forms
are completely determined by their action on $\cinfty{TQ}$ and on
$\forms{1}{Q}$. As such, the vertical and horizontal exterior
derivatives are determined by
\[
\dv{}F = V_i(F)dq^i, \qquad \dh{}F = H_i(F) dq^i, \qquad
F\in\cinfty{TQ},
\]
\[
\dv{}dq^i=0, \qquad \dh{}dq^i=0.
\]
Obviously, for $L\in\cinfty{TQ}$, $\dv{}L\in\forms{1}{\tau}$ has the
same coordinate representation as $S(dL)\in\forms{1}{TQ}$; in line
with the above remark therefore, we will also write
$\theta_L=\dv{}L$ for the Poincar\'e-Cartan 1-form. Derivations of
type $i_*$ are defined as in the standard theory. For
$A\in\Vforms{}{\tau}$, we put
\[
\dv{A} = [i^{}_A, \dv{}], \qquad \dh{A} = [i^{}_A, \dh{}],
\]
and call these derivations of type $\dv{*} $ and $\dh{*}$
respectively. The action of all such derivations can be extended to
vector-valued forms and then another algebraic type derivation is
needed for a classification, but we will introduce such extensions,
which can all be found in \cite{MaCaSaI, MaCaSaII}, only when
needed. The horizontal and vertical lift operations, already
referred to in the introduction, trivially extend to vector fields
along $\tau$ and then every vector field on $TQ$ has a unique
decomposition into a sum of the form $\H{X}+\V{Y}$, with
$X,Y\in\tvectorfields$. Looking in particular at the decomposition
of the commutator $[\H{X},\V{Y}]$ suffices to discover two important
derivations of degree zero:
\[
[\H{X},\V{Y}]= \V{(\DH{X}Y)} - \H{(\DV{Y}X)}.
\]
They extend to forms by duality and are called the horizontal and
vertical covariant derivatives. In coordinates
\begin{eqnarray*}
&& \DV{X} F = X^i\,V_i(F), \quad \DV{X}\fpd{}{q^i} = 0, \quad \DV{X}dq^i=0, \\[1mm]
&&\DH{X} F = X^i\,H_i(F), \quad \DH{X}\fpd{}{q^i} = X^j
V_j(\Gamma^k_i) \fpd{}{q^k}, \quad \DH{X}dq^i=-X^j V_j(\Gamma^i_k)
dq^k.
\end{eqnarray*}
For later use, we mention the following formulas for
computing exterior derivatives of, for example, a 1-form $\alpha$ or
a 2-form $\rho$ along $\tau$:
\begin{align}
 \dv{}\!\alpha\,(X,Y) &= \DV{X}\alpha \,(Y) - \DV{Y}\alpha \,(X), \qquad
\alpha\in \forms{1}{\tau}, \label{dv1form} \\ \dv{}\!\rho\,(X,Y,Z)
&= \sum_{X,Y,Z} \DV{X}\rho\,(Y,Z), \qquad \rho\in\forms{2}{\tau}
\label{dv2form},
\end{align}
and similarly for $\dh{}$.  Here $\sum_{X,Y,Z}$ represents the cyclic
sum over the indicated arguments.  It is also of interest to list the
decomposition of the other brackets of lifted vector fields:
\begin{align*}
[\V{X},\V{Y}] &= \V{\left(\DV{X}Y-\DV{Y}X\right)}, \\
{}[\H{X},\H{Y}] &= \H{\left(\DH{X}Y-\DH{Y}X\right)} +
\V{\left(R(X,Y)\right)}\,.
\end{align*}
The latter relation is just one of many equivalent ways in which the
curvature tensor $R\in\Vforms{2}{\tau}$ of the non-linear connection
can be defined. The connection coming from $\Gamma$ has no torsion
(since (\ref{conn}) obviously implies that
$V_i(\Gamma^j_k)=V_k(\Gamma^j_i)$): it follows that $\dv{}$ and
$\dh{}$ commute. In fact the commutation table of the exterior
derivatives, for their action on scalar forms, is given by
\begin{equation}
\onehalf [\dv{},\dv{}]= \dv{}\dv{} = 0, \qquad \dv{}\dh{}=
-\dh{}\dv{},\qquad \onehalf [\dh{},\dh{}]=\dh{}\dh{}= \dv{R}. \label{commutators}
\end{equation}

Finally, the given dynamical system $\Gamma$  comes canonically
equipped with two other operators which are crucial for our
analysis, namely the dynamical covariant derivative $\nabla$, a
degree zero derivation, and the Jacobi endomorphism
$\Phi\in\Vforms{1}{\tau}$. Again, the simplest way of introducing them
comes from the decomposition of a Lie bracket: they are the
uniquely determined operations for which, for each
$X\in\tvectorfields$,
\[
[\Gamma,\H{X}] = \H{(\del X)} + \V{(\Phi X)}.
\]
The usual duality rule $\del\langle X,\alpha\rangle = \langle \del
X,\alpha\rangle + \langle X,\del\alpha\rangle$ is used to extend the
action of $\del$ to 1-forms, and subsequently to arbitrary tensor
fields along $\tau$. In coordinates,
\begin{equation}
\del F = \Gamma(F),\quad \del\left(\fpd{}{q^j}\right) = \Gamma^i_j
\fpd{}{q^i}, \quad \del(dq^i) = - \Gamma^i_j dq^j\,, \label{del}
\end{equation}
and
\begin{equation}
\Phi^i_j = - \fpd{f^i}{q^j} - \Gamma^k_j \Gamma^i_k -
\Gamma(\Gamma^i_j) . \label{Phi}
\end{equation}
One clear indication of the importance of these operators is the
following link with the curvature of the connection:
\begin{equation}
\dv{}\Phi = 3\,R, \qquad \dh{}\Phi = \del R. \label{Phi-R}
\end{equation}

We are now ready to go back to the generalities about the inverse
problem discussed in the previous section. To begin with, using the
tools which have just been established, the compact formulation
(\ref{thetaL1}) of the inverse problem is equivalent (see
\cite{MaCaSaII}) to the existence of a regular function
$L\in\cinfty{TQ}$ such that
\begin{equation}
\del\theta_L = \dh{}L. \label{thetaL2}
\end{equation}
Secondly, the necessary and sufficient conditions (\ref{Mike}) now
really become conditions on the multiplier matrix; they are translated
via the K\"ahler lift \cite{MaCaSaII} into the existence of a
non-degenerate, symmetric $(0,2)$-tensor $g$ along $\tau$ satisfying
the requirements
\begin{equation}
\del g=0, \qquad g(\Phi X,Y)=g(X,\Phi Y), \qquad
\DV{X}g(Y,Z)=\DV{Y}g(X,Z). \label{Helmholtz2}
\end{equation}
It is possible to prove directly that (\ref{thetaL2}) implies
(\ref{Helmholtz2}) and vice versa (a sketch of such a proof was
presented in \cite{CaMa90}). We will not show how to do this here,
however, as it can easily be seen later on as a particular case of
the more general inverse problem studies we will start analyzing
now.

\section{Lagrangian systems with dissipative or gyroscopic forces}

Consider first equations of type (\ref{dissip1}). It is obvious
that, at the level of a characterization like (\ref{thetaL1}), a
given second-order field $\Gamma$ will correspond to equations of
type (\ref{dissip1}) if and only if there exist a regular function
$L$ and a function $D$ such that
\begin{equation}
\lie{\Gamma}(S(dL)) = dL + S(dD). \label{thetaLD1}
\end{equation}
We take the opportunity to illustrate first how such a relation, when
stripped to its bare essentials, i.e.\ when one observes that it is in
fact a condition on only $n$ of the $2n$ components, is transformed
into a corresponding generalization of (\ref{thetaL2}). To this end,
note first that there exists a dual notion of horizontal and
vertical lifts of 1-forms, from $\forms{1}{\tau}$ to
$\forms{1}{TQ}$, defined by $\H{\alpha}(\H{X})=\alpha(X),\
\H{\alpha}(\V{X})=0$, and likewise for $\V{\alpha}$. We then have
the following decompositions, for any $\alpha\in\forms{1}{\tau}$ and
$L\in\cinfty{TQ}$:
\begin{align}
\lie{\Gamma}\H{\alpha} &= \V{\alpha} + \H{(\del\alpha)}, \label{liealphah}\\
dL &= \H{(\dh{}L)} + \V{(\dv{}L)}. \label{dL}
\end{align}
Coming back to the notational remarks of the previous section: the
horizontal lift is technically speaking the rigorous way of
identifying a 1-form along $\tau$ with a semi-basic 1-form on $TQ$.
So, when convenient, as will be the case in establishing the next
result, we can also write $\theta_L=\H{(\dv{}L)}$, for example.

\begin{prop} \label{prop1}
The second-order field $\Gamma$ represents a dissipative system of
type (\ref{dissip1}) if and only if there is a regular function
$L\in\cinfty{TQ}$ and a function $D\in\cinfty{TQ}$ such that
\begin{equation}
\del\theta_L = \dh{}L + \dv{}D. \label{thetaLD2}
\end{equation}
\end{prop}

\begin{proof}
We have that $S(dD)=\H{(\dv{}D)}$ for any function $D$, and in
particular $S(dL)=\H{(\dv{}L)}= \theta_L$.  Using the decompositions
(\ref{liealphah}) and (\ref{dL}), the condition (\ref{thetaLD1}) then
immediately translates into (\ref{thetaLD2}).
\end{proof}

\begin{cor}
The condition (\ref{thetaLD2}) on the existence of functions $L$ and
$D$ is equivalent to
\begin{equation}
\dh{}\theta_L =0, \label{dhthetaL}
\end{equation}
which is a necessary and sufficient condition on $L$ only.
\end{cor}

\begin{proof}
Using the commutator property $[\del,\dv{}]= - \dh{}$ to re-express
$\del\theta_L=\del\dv{}L$ in (\ref{thetaLD2}), we immediately get
the expression
\[
\dh{}L= \onehalf\, \dv{}(\Gamma(L)-D),
\]
which is equivalent to saying that $\dh{}L=\dv{}G$ for some function
$G$. This in turn, in view of (\ref{commutators}) and the triviality
of $\dv{}$-cohomology, is equivalent to $\dh{}\theta_L= -
\dv{}\dh{}L=0$.
\end{proof}

We now want to translate these results into conditions on the
multiplier $g$ which generalize (\ref{Helmholtz2}).  As we observed
earlier, this $g$ will be the Hessian of $L$, so we look first
at the relation between a function and its Hessian in intrinsic terms.
To that end, we introduce covariant differentials $\DV{}$ and
$\DH{}$ defined as follows: for any tensor field $T$ along
$\tau$ and $X\in\tvectorfields$,
\[
\DV{}T(X,\ldots) = \DV{X}T(\ldots),
\]
and similarly for $\DH{}$.  Then for any $F\in\cinfty{TQ}$ we can
write for the corresponding Poincar\'e-Cartan 1-form $\theta_F =
\dv{}F=\DV{}F$, and define the Hessian tensor $g_F$ of $F$ as
$g_F=\DV{}\DV{}F$, which means that
\begin{equation}
g_F(X,Y) = \DV{X}\DV{Y}F - \DV{\DV{X}Y}F = \DV{X}\theta_F(Y).
\label{gF}
\end{equation}

\begin{lem} \label{lemma1} For any $F\in\cinfty{TQ}$, its
corresponding Hessian tensor $g_F$ is symmetric
and satisfies $\DV{X}g_F(Y,Z)=\DV{Y}g_F(X,Z)$, i.e.\ $\DV{}\!g_F$ is
symmetric in all its arguments.  Conversely, any symmetric $g$ along
$\tau$ for which $\DV{}\!g$ is symmetric is the Hessian of some
function $F$.  Secondly, if $\Phi$ represents any type $(1,1)$ tensor
field along $\tau$, we have
\begin{equation}
\Phi\hook g_F - (\Phi\hook g_F)^T =  i_{\dv{}\Phi}\theta_F -
\dv{}i_\Phi\theta_F, \label{Phi-g}
\end{equation}
where $(\Phi\hook g_F - (\Phi\hook g_F)^T)(X,Y) := g_F(\Phi X,Y) -
g_F(X,\Phi Y)$.
\end{lem}

\begin{proof}
The symmetry of $g_F$ follows directly from
$\dv{}\theta_F=\dv{}\dv{}F=0$. The symmetry of $\DV{}\!g_F$ can
easily be shown by taking a further vertical covariant derivative of
the defining relation of $g_F$ and using the commutator property
\begin{equation}
[\DV{X},\DV{Y}] = \DV{\DV{X}Y} - \DV{\DV{Y}X}. \label{DVDV}
\end{equation}
The converse statement is obvious from the coordinate representation
of the assumptions.  Finally, making use of (\ref{dv1form}) we have
\begin{align*}
\dv{}i_\Phi\theta_F (X,Y) &= \DV{X}(\Phi(\theta_F))(Y)
- \DV{Y}(\Phi(\theta_F))(X)  \\
&= \langle \DV{X}\Phi(Y) - \DV{Y}\Phi(X),\theta_F\rangle +
\DV{X}\theta_F(\Phi Y)
- \DV{Y}\theta_F(\Phi X) \\
&= i_{\dv{}\Phi}\theta_F(X,Y) + g_F(X,\Phi Y) - g_F(Y,\Phi X),
\end{align*}
from which the last statement follows.
\end{proof}

We are now ready to state and prove the first main theorem, which
provides the transition of the single condition (\ref{thetaLD2}) to
equivalent conditions involving a multiplier $g$, in precisely the
same way as (\ref{thetaL2}) relates to (\ref{Helmholtz2}).

\begin{thm} \label{thm1}
The second-order field $\Gamma$ represents a dissipative system of
type (\ref{dissip1}) if and only if there exists a function $D$ and
a symmetric type $(0,2)$ tensor $g$ along $\tau$ such that
$\DV{}\!g$ is symmetric and $g$ and $D$ further satisfy
\begin{align}
\del g &= \DV{}\DV{}D, \label{delgD} \\
\Phi\hook g - (\Phi\hook g)^T &= \dv{}\dh{}D, \label{Phi-gD}
\end{align}
where $\Phi$ is the Jacobi endomorphism of $\Gamma$.
\end{thm}

\begin{proof}
Suppose $\Gamma$ represents a system of type (\ref{dissip1}). Then
we know there exist functions $L$ and $D$ such that (\ref{thetaLD2})
and (\ref{dhthetaL}) hold true. Define $g=\DV{}\DV{}L$ or
equivalently $g(X,Y)=\DV{X}\theta_L(Y)$. Obviously, $g$ and
$\DV{}\!g$ are symmetric by construction. Acting with $\del$ on $g$
and using the commutator property $[\del,\DV{}]=-\DH{}$, we get for
a start
\begin{align*}
\del g &= \DV{}\del\DV{}L - \DH{}\DV{}L \\
&= \DV{}\DV{}D + \DV{}\DH{}L - \DH{}\DV{}L,
\end{align*}
where the last line follows from (\ref{thetaLD2}). Now the
commutator of vertical and horizontal covariant differentials (see
\cite{MaCaSaII} or \cite{CSMBP}) is such that, at least on
functions, $\DV{}\DH{}L(X,Y)=\DH{}\DV{}L(Y,X)$. But
\[
\DH{}\DV{}L(Y,X) - \DH{}\DV{}L(X,Y) = \dh{}\theta_L(Y,X) =0,
\]
in view of (\ref{dhthetaL}), so that (\ref{delgD}) follows. When
acting finally with $\dh{}$ on (\ref{thetaLD2}), we have to appeal
to the formula for $\dh{}\dh{}$ in (\ref{commutators}) and further
need the commutator of $\del$ and $\dh{}$, which for the action on
the module $\forms{}{\tau}$ of scalar forms is given by
\begin{equation}
{}[\del,\dh{}] = 2\,i_R + \dv{\Phi}. \label{deldh}
\end{equation}
It is then straightforward to check, using (\ref{dhthetaL}) and the
first of (\ref{Phi-R}), that we get
\begin{equation}
\dv{}i_\Phi\theta_L - i_{\dv{}\Phi}\theta_L = \dh{}\dv{}D, \label{Phi-D}
\end{equation}
from which (\ref{Phi-gD}) follows in view of the last statement in
Lemma~\ref{lemma1}.

Conversely, assume that $g$ and $D$ satisfy the four conditions
stated in the theorem. It follows from the symmetry of $g$ and
$\DV{}g$ that $g$ is a Hessian: $g=\DV{}\DV{}F$ say. The function
$F$ of course is not unique and the idea is to take advantage of the
freedom in $F$ to construct an $L$ which will have the desired
properties. This is not so difficult to do by a coordinate analysis.
Keeping the computations intrinsic is a bit more technical, but will
give us an opportunity to recall a few more features of interest of
the calculus of forms along $\tau$. Observe first that
$\nabla\DV{}g$ is obviously symmetric, and that the same is true
for $\DV{}\del g=\DV{}\DV{}\DV{}D$. It follows from
$[\del,\DV{}]=-\DH{}$ that $\DH{}g$ is also symmetric. Hence
\[
\DH{}g(X,Y,Z)=\DH{}\DV{}\theta_F(X,Y,Z)=\DH{}\DV{}\theta_F(X,Z,Y).
\]
If we interchange $\DH{}$ and $\DV{}$ in the last term, there is an
extra term to take into account (since the action is on a 1-form this time,
not a function). Indeed, we have
\begin{equation}
\DH{}\DV{}\theta_F(X,Z,Y)= \DV{}\DH{}\theta_F(Z,X,Y) +
\theta_F(\theta(X,Z)Y). \label{DHDV}
\end{equation}
Here $\theta$ is a type $(1,3)$ tensor along $\tau$ which is
completely symmetric (and could in fact be defined by the above
relation): its components in a coordinate basis are
$\theta^k_{jml}=V_mV_l(\Gamma^k_j)$. Using the above two relations,
expressing the symmetry of $\DH{}g$ in its first two arguments now
leads to
\[
0 = \DV{}\DH{}\theta_F(Z,X,Y) - \DV{}\DH{}\theta_F(Z,Y,X) =
\DV{Z}(\dh{}\theta_F)(X,Y).
\]
This says that $\dh{}\theta_F$ is a basic 2-form, i.e.\ a 2-form on
the base manifold $Q$. On the other hand, we have
\[
\DV{}\DV{}D = \del g=\del\DV{}\DV{}F = \DV{}\del\theta_F - \DH{}\DV{}F
=\DV{}(\del \theta_F - \dh{}F) - \dh{}\theta_F,
\]
where we have used the property $\DH{}\DV{}F(X,Y)=\DV{}\DH{}F(Y,X)$
again in the transition to the last expression. But since $\dh{}\theta_F$
is basic, we can write it as $\DV{}i_\T\dh{}\theta_F$, where $\T$ is
the canonical vector field along $\tau$ (the identity map on $TQ$),
which in coordinates reads
\begin{equation}
\T= v^i\fpd{}{q^i}. \label{bigT}
\end{equation}
It follows that we can write the last relation in the form
\[
\DV{}\beta:=\DV{}(\del\theta_F-\dh{}F - i_\T\dh{}\theta_F -
\DV{}D)=0,
\]
which defines another basic form $\beta$. We next want to prove that
the basic forms $\beta$ and $\dh{}\theta_F$ are actually closed in
view of the final assumption (\ref{Phi-gD}) or equivalently
(\ref{Phi-D}), which has not been used so far. Keeping in mind that
$\dh{}$ is the same as the ordinary exterior derivative for the
action on basic forms, we easily find with the help of
(\ref{commutators}) that
\[
\dh{}\dh{}\theta_F = \onethird \dv{}i_{\dv{}\Phi}\theta_F =
\onethird\dv{}(\dv{}i_\Phi\theta_F - \dv{}\dh{}D) =0.
\]
Secondly, using also (\ref{deldh}),
\begin{align*}
\dh{}\beta &=\dh{}\del\theta_F - \dh{}\dh{}F - \dh{}i_\T\dh{}\theta_F - \dh{}\dv{}D \\
&= \del\dh{}\theta_F - 2i_R\theta_F - \dv{\Phi}\theta_F - i_R\dv{}F - \dh{\T}\dh{}\theta_F - \dh{}\dv{}D \\
&= \del\dh{}\theta_F - i_{\dv{}\Phi}\theta_F + \dv{}i_\Phi\theta_F - \dh{\T}\dh{}\theta_F - \dh{}\dv{}D \\
&=\del\dh{}\theta_F - \dh{\T}\dh{}\theta_F.
\end{align*}
But this is zero also because the operators $\del$ and $\dh{\T}$
coincide when they are acting on basic (scalar) forms. It follows
that, locally, $\dh{}\theta_F = \dh{}\alpha$ and $\beta=\dh{}f$, for
some basic 1-form $\alpha$ and basic function $f$. The defining
relation for $\beta$ then further implies that
\[
\dv{}D = \del(\theta_F - \alpha) - \dh{}(F-i_\T\alpha + f).
\]
Putting $L=F-i_\T\alpha + f$, the difference between $L$ and $F$ is
an affine function of the velocities, so both functions have the
same Hessian $g$, and also $\theta_L = \DV{}L = \theta_F - \alpha$.
It now readily follows that the relation (\ref{thetaLD2}) holds
true, which concludes our proof in view of Proposition~\ref{prop1}.
\end{proof}

It is worthwhile listing the coordinate expressions for the
necessary and sufficient conditions of Theorem~\ref{thm1}. They call
for a (non-singular) symmetric matrix $g_{ij}(q,v)$ and a function
$D(q,v)$ such that
\begin{align}
V_k(g_{ij}) &= V_j(g_{ik}) \label{HD1} \\
\Gamma(g_{ij}) - g_{ik}\Gamma^k_j - g_{jk}\Gamma^k_i &= V_iV_j(D) \label{HD2} \\
g_{ik}\Phi^k_j - g_{jk}\Phi^k_i &= H_iV_j(D)-H_jV_i(D).
\label{HD3}
\end{align}
The classical Helmholtz conditions for the multiplier  are
recovered when we put $D=0$, of course.

Let us now turn to the case of forces of gyroscopic type as in
(\ref{dissip2}).

\begin{prop} \label{prop2} The second-order field $\Gamma$ represents
a gyroscopic system of type (\ref{dissip2}) if and only if there is a
regular function $L\in\cinfty{TQ}$ and a basic 2-form
$\omega\in\forms{2}{Q}$ such that
\begin{equation}
\del\theta_L = \dh{}L + i_\T \omega. \label{thetaLomega}
\end{equation}
\end{prop}

\begin{proof}
The proof is straightforward, by a simple coordinate calculation or an
argument like that in Proposition~\ref{prop1}.
\end{proof}

As a preliminary remark: it is easy to verify in coordinates that
for a basic 2-form $\omega$, we have
\begin{equation}
\DV{}i_\T\omega = \omega, \qquad \dv{}i_\T\omega = 2\,\omega, \qquad
\dv{}i_\T\dh{}\omega = 3\,\dh{}\omega. \label{aux}
\end{equation}
It follows by taking a vertical exterior derivative of
(\ref{thetaLomega}) that this time $\dh{}\theta_L$ will not vanish
but must be basic, specifically we must have
\begin{equation}
\dh{}\theta_L = \omega. \label{dhthetaLomega}
\end{equation}

\begin{thm} \label{thm2}
The second-order field $\Gamma$ represents a gyroscopic system of
type (\ref{dissip2}) if and only if there exists a basic 2-form
$\omega\in\forms{2}{Q}$ and a symmetric type $(0,2)$ tensor $g$
along $\tau$ such that $\DV{}\!g$ is symmetric and $g$ and $\omega$
further satisfy
\begin{align}
\del g &= 0, \label{delgomega} \\
\Phi\hook g - (\Phi\hook g)^T &= i_\T\dh{}\omega,
\label{Phi-gomega}
\end{align}
where $\Phi$ is the Jacobi endomorphism of $\Gamma$.
\end{thm}

\begin{proof}
Assuming we are in the situation described by
Proposition~\ref{prop2}, we define $g$ as before by $g=\DV{}\DV{}L$,
or $g(X,Y)=\DV{X}\theta_L(Y)=\DV{Y}\theta_L(X)$, from which the
usual symmetry of $\DV{}\!g$ follows. Acting with $\nabla$ on $g$
and following the pattern of the proof of Theorem~\ref{thm1}, we get
$\del g (X,Y) =\dh{}\theta_L(Y,X) + \DV{}i_\T\omega(X,Y)$, which is
zero in view of (\ref{aux}) and (\ref{dhthetaLomega}). Finally, for
the horizontal exterior derivative of (\ref{thetaLomega}), the
modifications are that the left-hand side produces a term
$\del\omega$ in view of (\ref{dhthetaLomega}), while the second term
on the right gives $\dh{}i_\T\omega = \dh{\T}\omega -
i_\T\dh{}\omega$, and since $\del=\dh{\T}$ on basic forms we end up
with the relation
\[
i_{\dv{}\Phi}\theta_L - \dv{}i_\Phi\theta_L = i_\T\dh{}\omega,
\]
which is the desired result (\ref{Phi-gomega}) in view of
Lemma~\ref{lemma1}.

For the sufficiency, we observe as before that $g$ is a Hessian, say
$g=\DV{}\DV{}F$, and that also $\DH{}g$ will be symmetric, which in
exactly the same way implies that $\dh{}\theta_F$ is basic. Still
following the pattern of Theorem~\ref{thm1}, $\del g=0$ will now
imply that $\beta:=\del\theta_F - \dh{}F - i_\T\dh{}\theta_F$ is a
basic 1-form. In computing $\dh{}\dh{}\theta_F$, the modification is
that
\[
\dh{}\dh{}\theta_F = \onethird \dv{}i_{\dv{}\Phi}\theta_F =
\onethird\dv{}(\dv{}i_\Phi\theta_F + i_\T\dh{}\omega) = \dh{}\omega,
\]
in view of the last of (\ref{aux}). Since $\dh{}\theta_F$ and
$\omega$ are basic, this expresses that their difference is closed
and thus locally exact: $\dh{}\theta_F= \omega + \dh{}\alpha$ for
some basic 1-form $\alpha$. The computation of $\dh{}\beta$ leads as
before to the conclusion that $\beta$ is closed, thus locally
$\beta=\dh{}f$ for some function $f$ on $Q$. Using this double
information, we find that
\[
i_\T\dh{}\theta_F = i_\T\omega + \del\alpha - \dh{}i_\T\alpha,
\]
and subsequently
\begin{align*}
0 &= \del \theta_F - \dh{}F - i_\T\dh{}\theta_F - \dh{}f \\
&= \del(\theta_F - \alpha) - \dh{}(F-i_\T\alpha + f) - i_\T\omega.
\end{align*}
This is a relation of type (\ref{thetaLomega}), with $L=F-i_\T\alpha +
f$, which concludes the proof.
\end{proof}

In coordinates, in comparison with the dissipative case of
Theorem~\ref{thm1}, the conditions (\ref{HD2}) and (\ref{HD3}) are
replaced in the gyroscopic case by
\begin{align}
\Gamma(g_{ij}) &= g_{ik}\Gamma^k_j + g_{jk}\Gamma^k_i  \label{Hg2} \\
g_{ik}\Phi^k_j - g_{jk}\Phi^k_i &= \onehalf
\left(\fpd{\omega_{ij}}{q^k} + \fpd{\omega_{jk}}{q^i} +
\fpd{\omega_{ki}}{q^j}\right)v^k \label{Hg3}
\end{align}
with $\omega_{ij}(q)=-\omega_{ji}(q)$.

Remark: when $d\omega=0$, the conditions of Theorem~\ref{thm2}
reduce to the standard Helmholtz conditions for a multiplier $g$.
This should not come as a surprise, since the local exactness of
$\omega$ then implies that the gyroscopic forces are actually of the
type of the Lorentz force of a magnetic field, for which it is known
that a generalized potential can be introduced to arrive at a
standard Lagrangian representation.

It is worth noting that in the sufficiency part of the proof the
condition $\del g=0$, given that $g$ and $\DV{}\!g$ are symmetric, is
used to show the existence of a basic 1-form $\beta$ such that
$\del\theta_F = \dh{}F + i_\T\dh{}\theta_F+\beta$, where
$\dh{}\theta_F$ is a basic 2-form.  The condition involving $\Phi$
then has the role of ensuring that $F$ can be modified by the addition
of a function affine in the fibre coordinates so as to eliminate the
$\beta$ term.  This suggests that it might be interesting to examine
the effect of ignoring the $\Phi$ condition.  When we do so we obtain
the following result.

\begin{prop} \label{prop2a}
For a given second-order field $\Gamma$, the existence of a
non-singulsr symmetric type $(0,2)$ tensor $g$ along $\tau$ such that
$\DV{}\!g$ is symmetric and $\del g=0$ is necessary and sufficient for
there to be a regular function $L$, a basic 1-form $\beta$ and a basic
2-form $\omega$ such that $\del\theta_L = \dh{}L + i_\T\omega+\beta$,
that is to say, such that the equations
\[
\frac{d}{dt}\left(\fpd{L}{\dot{q}^i}\right) - \fpd{L}{q^i} =
\omega_{ki}(q)\dot{q}^k+\beta_i(q), \quad \omega_{ki}= - \omega_{ik}.
\]
are equivalent to those determined by $\Gamma$.
\end{prop}

\begin{proof}
It remains to show that $\del g=0$ still holds when $\del\theta_L =
\dh{}L + i_\T\omega+\beta$. Since $\beta$ is basic,
$\DV{}\beta=0$, from which it follows easily that both of the
formulas  $\dh{}\theta_L = \omega$ and
$\del g (X,Y) =\dh{}\theta_L(Y,X) + \DV{}i_\T\omega(X,Y)$ continue to
hold, so that $\del g=0$ as before.
\end{proof}

One point of interest about this result is that it concerns a subset
of the full Helmholtz conditions.  Unlike Theorems 1 and 2 above, but
like the full Helmholtz conditions, it involves conditions on the
multiplier only, and in this respect it anticipates the results to be
found in the following section.

An analogous result has been obtained by different methods in \cite{olgageoff} (see Proposition~3.13 and the
immediately following remarks). This is the partial result that we
mentioned in the introduction.

\section{Reduction to conditions on the multiplier only}

We have seen in the previous section that Theorems 1 and 2 produce the
direct analogues of the Helmholtz conditions (\ref{Helmholtz2}) of the
standard inverse problem of Lagrangian mechanics.  It is quite natural
that the extra elements in our analysis, namely the function $D$,
respectively the 2-form $\omega$, make their appearance in these
covering generalizations.  Quite surprisingly, however, one can go a
step further in the generalizations and eliminate the dependence on
$D$ or $\omega$ all together, to arrive at necessary and sufficient
conditions involving the multiplier $g$ only.  This is what we will
derive now, but it is a rather technical issue, for which we will
therefore prepare the stage by proving a number of auxiliary results
first.  We recall that, as in the relation (\ref{dv2form}), a notation
like $\sum_{X,Y,Z}$ in what follows always refers to a cyclic sum
over the indicated arguments.

\begin{lem} \label{lemma2}
If $F\in\cinfty{TQ}$, $\theta_F = \dv{}F$ and $g=\DV{}\DV{}F$ then
\begin{align}
\dv{R}\theta_F(X,Y,Z) &= \sum_{X,Y,Z}g(R(X,Y),Z), \label{dvR} \\
\dh{R}\theta_F(X,Y,Z) &= \sum_{X,Y,Z}\DH{R(X,Y)}\theta_F(Z).
\label{dhR}
\end{align}
\end{lem}

\begin{proof}
In view of the fact that $\dv{}\dv{}=0$, $\dv{R}\theta_F$ reduces to
$\dv{}i_R\dv{}F$, and using (\ref{dv2form}) we then get
\begin{align*}
\dv{R}\theta_F(X,Y,Z)& =  \sum_{X,Y,Z} \DV{X}(i_R\dv{}F)(Y,Z)
= \sum_{X,Y,Z} (i_{\DV{X}R}\dv{}F + i_R \DV{X}\dv{}F)(Y,Z) \\
&= \sum_{X,Y,Z}  \Big[\dv{}F (\DV{X}R(Y,Z)) + g(X,R(Y,Z))\Big] \\
&= \sum_{X,Y,Z} g(R(X,Y),Z) + \dv{}F(\dv{}R(X,Y,Z)).
\end{align*}
Taking into account the fact that $3\dv{}R = \dv{}\dv{}\Phi=0$, the first
result follows.  For the second there are two terms to compute.  For
the first we have
\begin{align*}
\dh{}i_R\theta_F(X,Y,Z) &= \sum_{X,Y,Z} \DH{X}(i_R\theta_F)(Y,Z)
\\ &= \sum_{X,Y,Z} (i_{\DH{X}R}\theta_F + i_R \DH{X}\theta_F)(Y,Z)
= \sum_{X,Y,Z} \DH{X}\theta_F (R(Y,Z)),
\end{align*}
since the first term of the second line vanishes in view of the
Bianchi identity $\dh{}R=0$ \cite{MaCaSaI}. Secondly,
\[
i_R\dh{}\theta_F(X,Y,Z) = \sum_{X,Y,Z} \dh{}\theta_F(R(X,Y),Z) =
\sum_{X,Y,Z}\left[\DH{R(X,Y)}\theta_F(Z) -
\DH{Z}\theta_F(R(X,Y))\right].
\]
Adding these two expressions gives the desired result
(\ref{dhR}).
\end{proof}

\begin{lem} \label{lemma3}
If $\,\DH{}g$ is symmetric then
\begin{equation}
\dh{}(\Phi\hook g - (\Phi\hook g)^T) (X,Y,Z) = \sum_{X,Y,Z} g
(\nabla R (X,Y), Z). \label{DHg}
\end{equation}
\end{lem}

\begin{proof}
We have
\begin{eqnarray*}
\lefteqn{ \dh{}(\Phi\hook g - (\Phi\hook g)^T) (X,Y,Z) =
\sum_{X,Y,Z} \DH{X} (\Phi\hook g - (\Phi\hook g)^T) (Y,Z)} \\
&& \hspace*{-.2cm} = \sum_{X,Y,Z} ( \DH{X}\Phi\hook g + \Phi \hook
\DH{X} g -
(\DH{X}\Phi\hook g)^T -( \Phi \hook \DH{X} g)^T  )(Y,Z)\\
&& \hspace*{-.2cm} = \sum_{X,Y,Z} \Big[ g(\DH{X}\Phi(Y),Z) -
g(\DH{X}\Phi(Z),Y) \Big] + \sum_{X,Y,Z} \Big[ \DH{X}g(\Phi Y,Z) -
\DH{X}g(\Phi Z,Y) \Big].
\end{eqnarray*}
Making use of the cyclic sum freedom in the second and fourth term,
and of the symmetry of $\DH{}g$ in the third, the right-hand side
reduces to
\[
\sum_{X,Y,Z}  g(\DH{X}\Phi(Y) - \DH{Y}\Phi(X),Z) =   \sum_{X,Y,Z}
g(\dh{}\Phi(X,Y),Z),
\]
which proves our statement in view of $\dh{}\Phi=\nabla R$.
\end{proof}

\begin{lem} \label{lemma4}
If $g$ and $\DV{}g$ are both symmetric then
\begin{eqnarray*}
\lefteqn{[\nabla,\DH{}]g\,(X,Y,Z) - [\nabla,\DH{}]g\,(Y,X,Z)} \\
&& \hspace*{3cm}= \DV{Z} (\Phi\hook g - (\Phi\hook g)^T)(X,Y) -
\sum_{X,Y,Z} g(R(X,Y),Z).
\end{eqnarray*}
\end{lem}

\begin{proof}
The commutator $[\del,\DH{}]$ is rather complicated when it comes to
its action on a symmetric type $(0,2)$ tensor $g$. It reads (see for
example \cite{CSMBP} where it was already used):
\begin{align*}
[\del,\DH{}]g\,(X,Y,Z) &= \DV{\Phi X}g(Y,Z) - 2g(R(X,Y),Z) -2g(R(X,Z),Y)  \\
&\qquad\mbox{} + g(\DV{X}\Phi(Y),Z) + g(\DV{X}\Phi(Z),Y).
\end{align*}
Subtracting the same expression with $X$ and $Y$ interchanged, it is
however a fairly simple computation, using the symmetry of $\DV{}g$
and the property $\dv{}\Phi=3R$, to arrive at the desired result.
\end{proof}

\begin{lem} \label{lemma5}
For all $F\in\cinfty{TQ}$ we have
\begin{equation}
\DH{}\DV{}\DV{}F (X,Y,Z) - \DH{}\DV{}\DV{}F (Y,X,Z) =
\DV{Z}\dh{}\dv{}F (X,Y). \label{DHDV2}
\end{equation}
\end{lem}

\begin{proof}
This is in fact a variation of a formula which was already used in
proving that $\dh{}\theta_F$ is basic in the second part of the
proof of Theorem~\ref{thm1}. We have to appeal again to the general
formula (\ref{DHDV}), applied to $\DV{}F=\dv{}F=\theta_F$. After
swapping the last two arguments in each term on the left in
(\ref{DHDV2}), a direct application of this formula easily leads to
the result.
\end{proof}

\begin{thm} \label{thm3}
The second-order field $\Gamma$ represents a dissipative system of
type (\ref{dissip1}) if and only if there exists a symmetric type
$(0,2)$ tensor $g$ along $\tau$ such that both $\DV{}g$ and
$\DH{}g$ are symmetric and
\begin{equation}
\sum_{X,Y,Z}g(R(X,Y),Z)=0. \label{Rcond}
\end{equation}
\end{thm}

\begin{proof}
Assume first that the conditions of Theorem~\ref{thm1} hold true. So
$g$ and $\DV{}g$ are symmetric and as before, since $\del$ preserves
the symmetry of $\DV{}g$ and also $\DV{}\del g= \DV{}\DV{}\DV{}D$ is
manifestly symmetric, we conclude that $\DH{}g$ is symmetric.
Moreover, if $F$ is any function such that $g=\DV{}\DV{}F$, we know
from Lemma~\ref{lemma1} that
\[
\Phi\hook g - (\Phi\hook g)^T = i_{\dv{}\Phi}\theta_F -
\dv{}i_\Phi\theta_F.
\]
It then follows from the last condition in Theorem~\ref{thm1} that
\[
0=\dv{}\dv{}\dh{}D = \dv{}i_{\dv{}\Phi}\theta_F = 3\,\dv{R}\theta_F,
\]
so that the first statement in Lemma~\ref{lemma2} implies
(\ref{Rcond}).

For the converse, symmetry of $g$ and $\DV{}g$ imply that $\del g$
and $\del\DV{}g$ are symmetric, and since in addition $\DH{}g$ is
symmetric, we conclude that $\DV{}\del g$ is symmetric, which means
that $\del g$ is also a Hessian (see Lemma~\ref{lemma1}), say
$\del g=\DV{}\DV{}D$ for some function $D$.  Next, we look at the
statement of Lemma~\ref{lemma4} in which the last term vanishes here
by assumption.  We have that $\del\DH{}g$ is symmetric, so that the
left-hand side reduces to
\[
-\DH{}\DV{}\DV{}D(X,Y,Z) + \DH{}\DV{}\DV{}g(Y,X,Z).
\]
Combining the results of Lemma~\ref{lemma4} and
Lemma~\ref{lemma5} we conclude that the 2-form
\[
\beta:= \Phi\hook g - (\Phi\hook g)^T + \dh{}\dv{}D
\]
is basic. Now from the last of the properties (\ref{commutators})
and Lemma~\ref{lemma2} applied to $\del g$, which is determined by
$\theta_D$, we know that
\[
\dh{}\dh{}\dv{}D = \dv{R}\theta_D = \sum_{X,Y,Z} \del g (R(X,Y),Z).
\]
This in turn, making use also of the result of Lemma~\ref{lemma3},
gives rise to the following calculation:
\begin{align*}
\dh{}\beta &= \sum_{X,Y,Z}g(\del R(X,Y),Z) + \sum_{X,Y,Z} \del g
(R(X,Y),Z)\\
&= \del\Big(\sum_{X,Y,Z}g(R(X,Y),Z)\Big) \\
&\qquad \mbox{} - \sum_{X,Y,Z}g(R(\del X,Y),Z) - \sum_{X,Y,Z}g(R(X,\del
Y),Z) - \sum_{X,Y,Z}g(R(X,Y),\del Z) \\
&= - \sum_{X,Y,Z}\left[g(R(\del Z,X),Y) + g(R(Y,\del Z),X) +
g(R(X,Y),\del Z)\right].
\end{align*}
The expression between square brackets in the last line is zero
because of (\ref{Rcond}), with vector arguments $X,Y$ and $\del Z$;
it follows that $\beta$ is closed, thus locally $\beta=\dh{}\alpha$
for some basic 1-form $\alpha$. Putting $\widetilde{D}= D -
i_\T\alpha$, we have $\del g=
\DV{}\DV{}D=\DV{}\DV{}\widetilde{D}$, and $\beta-\dh{}\dv{}D = -
\dh{}\dv{}\widetilde{D}$, so that $\Phi\hook g - (\Phi\hook g)^T =
\dv{}\dh{}\widetilde{D}$ and all conditions of Theorem~\ref{thm1}
are satisfied.
\end{proof}

The results of Theorem~\ref{thm3} deserve some further comments.
Establishing necessary and sufficient conditions for the existence
of a Lagrangian is in a way the easy part of the inverse problem; the hard
part is the study of formal integrability of these conditions, for
which a number of different techniques exist (see for example
\cite{AT92}, \cite{GM1, GM2}, \cite{SaCraMa}). If we go back to the
standard Helmholtz conditions (\ref{Helmholtz2}), for example, two
of the first integrability conditions one encounters are the
symmetry of $\DH{}g$ and the algebraic condition (\ref{Rcond}). So
in the standard situation, if a $g$ exists satisfying
(\ref{Helmholtz2}), these properties will automatically hold true:\
it seems to us noteworthy that these two integrability
conditions make their appearance in the dissipative case as part of
the starting set of necessary and sufficient conditions. It is
further worth observing that the case of Rayleigh dissipation can be
characterized by the further restriction that $\DV{}\del g=0$.
Indeed, since $\del g=\DV{}\DV{}D$, this extra condition will imply
that $D$ must be quadratic in the velocities.

The coordinate expressions of the conditions in
Theorem~\ref{thm3} are, apart from (\ref{HD1}),
\begin{align}
H_{i}(g_{jk}) - H_{j}(g_{ik}) + g_{il} \conn{l}{j}{k} - g_{jl} \conn{l}{i}{k} &=0 \\
g_{ij}R^j_{kl} + g_{lj}R^j_{ik} + g_{kj}R^j_{li} &=0,
\end{align}
where $\conn{l}{j}{k} = V_k(\Gamma^l_j)$ and $R^k_{ij}=
H_j(\Gamma^k_i) - H_i(\Gamma^k_j) =
\onethird(V_i(\Phi^k_j)-V_j(\Phi^k_i))$.

We now proceed in the same way for the gyroscopic case.

\begin{thm} \label{thm4}
If the second-order field $\Gamma$ represents a gyroscopic system of
type (\ref{dissip2}) then there exists a symmetric type $(0,2)$
tensor $g$ along $\tau$ such that $\DV{}g$ is symmetric, $\del g=0$
and
\begin{equation}
\left(\Phi\hook g - (\Phi\hook g)^T\right)(X,Y)= \sum_{X,Y,\T}g(R(X,Y),\T).
\label{Rcond2}
\end{equation}
The converse is true as well, provided we assume that $\Phi\hook g$
is smooth on the zero section of $TQ\rightarrow Q$.
\end{thm}

\begin{proof}
Assume we have a $g$ and $\omega$ satisfying the conditions of
Theorem~\ref{thm2}. Acting on the condition (\ref{Phi-gomega}) with
$\dv{}$, the left-hand side reduces, as in the proof of the
preceding theorem, to $3\dv{R}\theta_F$ for any $F$ such that
$g=\DV{}\DV{}F$. For the right-hand side, we get
$\dv{}i_\T\dh{}\omega = 3\dh{}\omega$. Hence
$\dh{}\omega=\dv{R}\theta_F$, and (\ref{Phi-gomega}) can be written
as
\begin{equation}
\Phi\hook g - (\Phi\hook g)^T = i_\T\dv{R}\theta_F.
\label{PhithetaF}
\end{equation}
Making use of Lemma~\ref{lemma2} the result now immediately follows.

Conversely, (\ref{Rcond2}) obviously implies (\ref{PhithetaF}) for
any $F$ such that $g=\DV{}\DV{}F$. $\DV{}g$ symmetric and $\del g=0$
imply that $\DH{}g$ is symmetric and then also $[\del,\DH{}]g$ is
symmetric. It follows from Lemma~\ref{lemma4} that
\[
\DV{Z} (\Phi\hook g - (\Phi\hook g)^T)(X,Y) = \sum_{X,Y,Z}
g(R(X,Y),Z), \qquad \forall X,Y,Z.
\]
In particular, taking $Z$ to be $\T$ and using Lemma~\ref{lemma2}
again plus (\ref{PhithetaF}), we obtain
\[
\DV{\T} (\Phi\hook g - (\Phi\hook g)^T) = \Phi\hook g - (\Phi\hook
g)^T.
\]
This asserts that $\Phi\hook g - (\Phi\hook g)^T$ is homogeneous
of degree 1 in the fibre coordinates. The additional smoothness
assumption then further implies linearity in the fibre coordinates,
so that there exists a basic 3-form $\rho$ such that $\Phi\hook g -
(\Phi\hook g)^T = i_\T \rho$. There are two conclusions we can
draw from this by taking appropriate derivatives. On the one hand,
taking the horizontal exterior derivative and using
Lemma~\ref{lemma3} we obtain
\[
\sum_{X,Y,Z} g (\nabla R (X,Y), Z) = (\dh{}i_\T\rho)(X,Y,Z) =
(\del\rho - i_\T\dh{}\rho)(X,Y,Z).
\]
On the other, knowing that $\DV{Z}i_\T\rho= i_Z\rho$ for any $Z$ and
appealing once more to the general conclusion of Lemma~\ref{lemma4},
we see that actually $\rho(X,Y,Z)=\sum_{X,Y,Z}g(R(X,Y),Z)$, from
which it follows in view of $\del g=0$ that
$\del\rho(X,Y,Z)=\sum_{X,Y,Z}g(\del R(X,Y),Z)$. The conclusion from
the last displayed equation is that $i_\T\dh{}\rho=0$. But then
$0=\DV{X}i_\T\dh{}\rho = i_X\dh{}\rho$ for all $X$, so that
$\dh{}\rho=0$ and locally $\rho=\dh{}\omega$ for some basic
$\omega$. It follows that
\[
\Phi\hook g - (\Phi\hook g)^T = i_\T\dh{}\omega,
\]
which completes the proof.
\end{proof}

In contrast with the preceding theorem, the condition (\ref{Rcond2})
which makes its appearance here is not one which is directly
familiar from the integrability analysis of the standard Helmholtz
conditions. But indirectly, when $\omega=0$, the left-hand side
vanishes and the fact that this is also the case for the right-hand
side follows from the integrability condition (\ref{Rcond}).

The coordinate expressions of the conditions in Theorem~\ref{thm4}, in
addition to (\ref{HD1}) and (\ref{Hg2}), are
\begin{equation}
g_{lj}\Phi^j_k - g_{kj}\Phi^j_l =(g_{ij}R^j_{kl} + g_{lj}R^j_{ik} +
g_{kj}R^j_{li})v^i.
\end{equation}

Before embarking on examples, it is worth emphasizing a fundamental
advantage of our intrinsic approach: we are not restricted to the
coordinate expressions in natural bundle coordinates listed so far, if
there are good reasons to work in a non-standard frame.  This is the
case, for example, in applications where it is appropriate to work
with so-called quasi-velocities.  Quasi-velocities are just fibre
coordinates in $TQ$ with respect to a non-standard frame $\{{X_i}\}$
of vector fields on $Q$ (which also constitute a basis for the module
of vector fields along $\tau$).  All conditions we have encountered so
far may be projected onto such a frame and rewritten in terms of the
quasi-velocities.  For example, take the condition (\ref{thetaLD2}) we
started from in the preceding section.  It can be expressed as
follows:
\begin{align*}
0 &= (\nabla\theta_L - \dh{} L - \dv{} D) (X_i) \\
&=\Gamma(\theta_L(X_i)) - \theta_L(\nabla X_i) - \hlift{X}_i (L) - \vlift{X}_i( D) \\
&= \Gamma (\vlift{X}_i(L)) - \vlift{(\nabla X_i)} (L) -\hlift{X}_i (L) - \vlift{X} _i (D)\\
& = \Gamma (\vlift{X}_i(L)) -\clift{X}_i (L) - \vlift{X} _i (D),
\end{align*}
where $\clift{X}_i$ stands for the complete lift of the vector field
$X_i$. Quasi-velocities $w^i$ can be thought of as the components of
$\T$ with respect to some anholonomic frame $\{{X_i}\}$ of vector
fields on $Q$. One can show (see e.g.\ \cite{nonholvak}) that the
complete and vertical lifts of such a frame, expressed in the
coordinates $(q,w)$, take the form
\[
\clift{X_i}=X_i^j\fpd{}{q^j}-A^j_{ik}v^k\fpd{}{w^j},\quad
\vlift{X_i}=\fpd{}{w^i},
\]
where $X_i=X_i^j\partial/\partial q^j$ and $[X_i,X_j]=A_{ij}^k X_k$.
The condition (\ref{thetaLD2}) now becomes
\[
\Gamma\left(\fpd{L}{w^i}\right)-X_i^j\fpd{L}{q^j}+A^j_{ik}w^k\fpd{L}{w^j}=
\fpd{D}{w^i}.
\]
These are the Boltzmann-Hamel equations referred to in \cite{germ2},
where, since the results the same authors obtained in \cite{germ1}
were expressed only in standard coordinates, all conditions had to be
rederived from scratch.  Needless to say, one can also recast any of
the other coordinate-free conditions we have obtained in terms of
quasi-velocities.

\section{Illustrative examples and concluding remarks}

We start with a simple linear system with two degrees of freedom,
which will serve us well to illustrate a number of features of the
results we have obtained. Consider the system
\begin{align}
\ddot{q}_1 &= -aq_1 -bq_2 - \omega\dot{q}_1, \label{ex11} \\
\ddot{q}_2 &= bq_1 - aq_2 + \omega\dot{q}_2, \label{ex12}
\end{align}
where $a,b$ and $\omega$ are constant, non-zero parameters. The only
non-zero connection coefficients are
\[
\Gamma^1_1 = \onehalf \omega = - \Gamma^2_2,
\]
and we obtain
\[
\Phi^1_1 = \Phi^2_2 = a - \onefourth \omega^2, \qquad \Phi^1_2
=b=-\Phi^2_1.
\]
Since $\Phi$ is constant, the curvature tensor $R$ is zero so that
condition (\ref{Rcond}) is satisfied (in fact it is void anyway in
view of the dimension). It follows from Theorem~\ref{thm3} that any
constant symmetric $g$ should be a multiplier for a representation
of the given system in the form (\ref{dissip1}). We consider three
such non-singular matrices:
\[
g^{(1)} = \left(\begin{array}{cc} 1&0\\ 0&-1 \end{array}\right)
\qquad g^{(2)} = \left(\begin{array}{cc} 0&1\\ 1&0
\end{array}\right) \qquad g^{(3)} = \left(\begin{array}{cc} 1&0\\ 0&1
\end{array}\right).
\]
For $g^{(1)}$, it is easy to verify that with
\begin{align*}
L_1 &= \onehalf (\dot{q}_1^2 - \dot{q}_2^2) - \onehalf a(q_1^2 -
q_2^2) - bq_1q_2, \\
D_1 &= -\onehalf\omega(\dot{q}_1^2 + \dot{q}_2^2),
\end{align*}
we have a representation of the given system in the form
(\ref{dissip1}). In the case that $\omega=0$, $g^{(1)}$ is still a
multiplier for the standard inverse problem and $L_1$ then becomes a
genuine Lagrangian. Also $g^{(2)}$, which changes the order of the
equations, is a multiplier in that case, leading to an alternative
Lagrangian for the same reduced system. But that Lagrangian cannot
serve for a dissipative representation of the full system. Instead,
we have to take
\[
L_2 = \dot{q}_1\dot{q}_2 - a q_1q_2 - \onehalf b (q_2^2 - q_1^2) +
\onehalf\omega (q_1\dot{q_2} - q_2\dot{q}_1),
\]
and then $D_2=0$. We discover here that the given system is
variational, with $L_2$ as Lagrangian. For $g^{(3)}$ the situation
is different again. This time, this is not a multiplier for the
reduced system ($\omega=0$), it violates the condition that
$\Phi\hook g$ must be symmetric. But for the full system, we can
simply take a kinetic energy Lagrangian and then make a suitable
adaptation for $D$. Explicitly,
\begin{align*}
L_3 &= \onehalf(\dot{q}_1^2 + \dot{q}_2^2), \\
D_3 &= - a(q_1\dot{q}_1 + q_2\dot{q}_2) + b(q_1\dot{q}_2 -
q_2\dot{q}_1) + \onehalf \omega (\dot{q}_2^2 - \dot{q}_1^2).
\end{align*}

Let us now look at the same system from the gyroscopic point of
view. Since $R=0$, the rather peculiar condition (\ref{Rcond2}) of
Theorem~\ref{thm4} reduces to the usual $\Phi$-condition and
Theorem~\ref{thm4} simply states the standard Helmholtz conditions
for the existence of a multiplier. In other words, any multiplier
for a representation in the form (\ref{dissip2}) will be a
multiplier for a variational description as well. Of the
non-singular, constant $g^{(i)}$ we considered before, only
$g^{(2)}$ satisfies the conditions now, and we can take
\[
L_4 = \dot{q}_1\dot{q}_2 - a q_1q_2 - \onehalf b (q_2^2 - q_1^2),
\]
with the 2-form $\omega\, dq_1\wedge dq_2$ to satisfy the
requirements of Theorem~\ref{thm2}. It should of course not come as
a surprise that we must have a variational formulation here as well,
since we are in the situation described in the remark towards the end
of Section~3. In fact we have already found the Lagrangian for this
variational formulation: it is the function $L_2$.

For a second example, with $n=3$, consider the non-linear system
\begin{align}
\ddot{q}_1&= q_2 \dot{q}_1 \dot{q}_3,  \label{ex21}\\
\ddot{q}_2&= \dot{q}_3^2,  \label{ex22}\\
\ddot{q}_3&=  \dot{q}_1^2 - q_2^{-1}{\dot q}_2{\dot q}_3.
\label{ex23}
\end{align}
From (\ref{Phi}) one easily verifies that
\[
\Big(\Phi^i_j\Big) = \left( \begin{array}{ccc}
- \onefourth q_2^2\dot{q}_3^2 & - \threefourth \dot{q}_1\dot{q}_3 &
\onefourth q_2^2\dot{q}_1\dot{q}_3 + \threefourth \dot{q}_1\dot{q}_2 \\[1mm]
-\dot{q}_1\dot{q}_3 & \onehalf q_2^{-1}\dot{q}_3^2 &
-\onehalf q_2^{-1}\dot{q}_2\dot{q}_3 + \dot{q}_1^2 \\[1mm]
\onehalf q_2\dot{q}_1\dot{q}_3 + \onehalf q_2^{-1}\dot{q}_1\dot{q}_2
& -\onefourth q_2^{-2}\dot{q}_2\dot{q}_3 - \onehalf
q_2^{-1}\dot{q}_1^2 & -\onehalf q_2\dot{q}_1^2 + \onefourth
q_2^{-2}\dot{q}_2^2
 \end{array} \right)
\]
and the curvature tensor $R=\onethird \dv{}\Phi$ is given by
\begin{align*}
R &= -\Big(\oneeights {\dot q}_3 dq_1\wedge dq_2 -(\oneeights{\dot
q}_2+\oneeights q_2^2{\dot q}_3 ) dq_1\wedge dq_3 - \onefourth {\dot
q}_1 dq_2\wedge dq_3\Big)\otimes \fpd{}{q_1} \\
&\qquad \mbox{}+\Big( \onehalf{\dot q}_1 dq_1\wedge dq_3
- \onefourth {q_2^{-1}}{\dot q}_1
dq_2\wedge dq_3\Big)\otimes \fpd{}{q_2}\\
&\qquad \mbox{}-\Big(\onefourth{q_2^{-1}}{\dot q}_1 dq_1\wedge dq_2
+\onefourth q_2{\dot q}_1
dq_1\wedge dq_3 - \oneeights q_2^{-2}{\dot q}_2 dq_2\wedge dq_3\Big)\otimes \fpd{}{q_3}
\end{align*}
The multiplier problem is already quite complicated for a system of
this kind and it is not our intention here to explore all possible
solutions. For simplicity, therefore, we limit ourselves in the
dissipative case (\ref{dissip1}) to analyzing the existence of a
diagonal multiplier $g$ which depends on the coordinates $q_i$ only.
With such an ansatz, the curvature condition (\ref{Rcond}) in
Theorem~\ref{thm3} reduces to
\[
g_{33}=(g_{11} - 2 g_{22})q_2,
\]
and the requirement that $\DH{}g$ should be symmetric subsequently
imposes that $g_{11}=4\,g_{22}=\mbox{constant}$. Hence, up to a
constant factor, we are reduced to the possibility that
\[
g_{11}=4, \qquad g_{22}=1, \qquad g_{33}= 2q_2.
\]
As was mentioned in the previous section, the conditions imposed so
far are also integrability conditions in the standard inverse
problem so that, starting from the same ansatz, this $g$ would also
be the only candidate for a standard Lagrangian representation of
the system. When we compute $\del g$ now, we get
\[
\Big( (\del g)_{ij} \Big) = \left( \begin{array}{ccc}
4q_2\dot{q}_3 & 0 & 4q_2\dot{q}_1 \\
0 & 0 & 0 \\
4q_2\dot{q}_1 & 0 & 0 \end{array} \right).
\]
Since $\del g \neq 0$, our candidate cannot lead to a variational
formulation. On the other hand, Theorem~\ref{thm3} is satisfied, so
there must exist a $D$ for a dissipative representation. From the
requirement (\ref{delgD}) in Theorem~\ref{thm1}, one easily verifies
that such a $D$ must satisfy
\[
V_1(D) = 4q_2\dot{q}_1\dot{q}_3 + h_1, \qquad V_2(D) = h_2, \qquad
V_3(D)= 2 q_2 \dot{q}_1^2 + h_3,
\]
where the $h_i$ are as yet arbitrary functions of the coordinates.
The final requirement (\ref{Phi-gD}) of Theorem~\ref{thm1} then
shows that the $h_i$ can be taken to be zero. Thus,
\[
L= \onehalf \Big( 4\dot{q}_1^2 + \dot{q}_2^2 + 2q_2\dot{q}_3^2\Big)
\quad \mbox{and} \quad D= 2q_2\dot{q}_1^2 \dot{q}_3,
\]
provide a solution for the inverse problem of type (\ref{dissip1})
for the given system.

Concerning the inverse problem of type (\ref{dissip2}), it is less
appropriate to look for a diagonal $g$ (as the example with $n=2$
has shown), but even if we extend our search to a general $g$
depending on the $q_i$ only, the conditions of Theorem~\ref{thm4}
have no non-singular solution.

Consider, finally, the system
\begin{align}
\ddot{q}_1&=  b \dot{q}_1 \dot{q}_4,  \label{ex31}\\
\ddot{q}_2&= \dot{q}_2 \dot{q}_4 ,  \label{ex32}\\
\ddot{q}_3&=  (1 - b) {\dot q}_1 {\dot q}_2 + b q_2 {\dot q}_1{\dot q}_4
- b q_1 {\dot q}_2 {\dot q}_4 + (b + 1) {\dot q}_3 {\dot q}_4, \label{ex33}\\
\ddot{q}_4&=0, \label{ex34}
\end{align}
with $-1<b<1$ and $b\neq 0$. These equations can be interpreted as
the geodesic equations of the canonical connection associated with a
certain Lie group $G$, which is uniquely defined by $\nabla_XY
=\onehalf[X,Y]$, where $X$ and $Y$ are left-invariant vector fields.
In the case of the above system, the Lie group is listed as
$A_{4,9b}$ in \cite{Gerard4}, and it was shown (see also
\cite{APST}) that the system does not have a variational
formulation. This is a consequence of the integrability condition
(\ref{Rcond}) which can only be satisfied by multipliers for which
 $g_{13}=g_{23}=g_{33}=0$. But then, the
$\Phi$-condition in (\ref{Helmholtz2}) leads
automatically to $g_{34}=0$ so that there is no non-singular solution.

Notice that the system is invariant for translations in the $q_3$
and $q_4$ direction; it is therefore reasonable that we limit
ourselves in our search for non-conservative representations to
multipliers with the same symmetry. In the dissipative case, after
haven taken the same curvature condition into account, the
$\DH{}g$-condition leads to the further restrictions
\[
\fpd{g_{34}}{{\dot q}_3} = 0, \qquad \fpd{g_{34}}{{\dot q}_3}-
g_{34}=0,
\]
among others, from which again $g_{34}=0$ follows, with the same negative
conclusion. In the gyroscopic case, one can show that the condition
(\ref{Rcond2}) cannot be satisfied for a multiplier with
coefficients depending on the coordinates $q_i$ only.

Some final comments are in order. In the
case of linear systems such as our first example, it frequently
happens that a multiplier for the inverse problem exists which is a
function of time only (see for example \cite{Sarlet80}), so we will
briefly sketch here how our present theory can be extended to
general, potentially time-dependent second-order systems. First of
all, the extension of the calculus along $\tau:TQ \rightarrow Q$ to
a time-dependent setting has been fully developed in
\cite{SaVaCaMa}, which also contains the analogues of
the conditions (\ref{thetaL2}) and (\ref{Helmholtz2}) for the
inverse problem. In all generality, we are then talking about a
calculus of forms along the projection $\pi:\R \times TQ \rightarrow
\R\times Q$ say. But as has been observed for example in \cite{CSMBP}, the
extra time-component in this setting does not really play a role
when it comes to studying the Helmholtz conditions and their
integrability. That is to say: one has to use $dt$ and the contact
forms $dq^i-\dot{q}^i dt$ as local basis for forms along $\pi$ and a
suitable dual basis for the vector fields which includes the given
second-order system $\Gamma$; important geometrical objects such as
the Jacobi endomorphism $\Phi$ will pick up an extra term for sure,
but when restricted to act on vector fields without
$\Gamma$-component, all formulas of interest formally look the same.
It is therefore not so hard to apply a suitably reformulated version
of the present theory to time-dependent systems, when needed. That
important formulas formally look the same will be seen also in the
final observations in the appendix, where the setting is essentially
time-dependent, though the approach adopted there is quite different
again from the calculus along $\pi$ we are referring to here.

\section*{Appendix}

In this appendix we will relate our results to those obtained by
Kielau et al.\ in \cite{germ2,germ1}, especially the latter; but first
we wish to derive those results anew, in a way which allows us to
explain an interesting feature of them which was mentioned in
\cite{germ1} but not fully dealt with there.

The problem discussed in \cite{germ2,germ1} differs in several ways
from the one which has been the subject of our paper, the most
important of which is that it is assumed there that a system of
second-order ordinary differential equations is given in implicit
form $f_i(t,q,\dot{q},\ddot{q})=0$, and the problem posed is to find
necessary and sufficient conditions on the functions $f_i$ such that
the equations may be written in the form (\ref{dissip1}), where $L$
and $D$ are allowed to be time-dependent. That is, the question is
whether the equations are of Lagrangian type with dissipation as
they stand, rather then whether they may be made equivalent to such
equations by a choice of multiplier.

It is probably most satisfactory to approach the inverse problem for a
second-order system given in implicit form by using the methods
associated with variational sequences, rather
than the techniques employed in the body of the paper.  Fortunately we
will need only the rudiments of such methods, one version of which we
now briefly describe; justification for the unsupported claims we make
can be found in \cite{vitolo}, for example.

We deal with the (trivial) fibred manifold $\pi:Q\times\R\to\R$, and
its infinite jet bundle $J^\infty(\pi)$; this may seem a bit
extravagant when we are interested only in second-order equations, but
is convenient for technical reasons.  However, all functions and forms
under consideration will be of finite type (i.e.\ depend on finitely
many variables).  We take coordinates $(t,q^i)$ on $\R\times Q$; the
jet coordinates are written $\dot{q}^i$, $\ddot{q}^i$ and so on. We
denote the contact 1-forms by
\[
\theta^i=dq^i-\dot{q}^i dt,\quad
\dot{\theta}^i=d\dot{q}^i-\ddot{q}^idt,\quad
\ddot{\theta}^i=d\ddot{q}^i-\dddotqi dt,\quad\ldots.
\]
We need two exterior-derivative-like operators on exterior forms on
$J^\infty(\pi)$. The first is the vertical differential $d_V$ (not to
be confused with $\dv{}$), which is defined by
\[
d_Vf=\fpd{f}{q^i}\theta^i+\fpd{f}{\dot{q}^i}\dot{\theta}^i+
\fpd{f}{\ddot{q}^i}\ddot{\theta}^i+\cdots,\quad
d_V dt=d_V \theta^i=d_V \dot{\theta}^i=d_V \ddot{\theta}^i=\ldots=0.
\]
The key properties of $d_V$ are that $d_V^2=0$, and that $d_V$ is
locally exact.  The second operator is the variational differential
$\delta$, about which we need to say just the following.  First, for a
Langrangian $L(t,q,\dot{q})$
\[
\delta(Ldt)=E_i(L)\theta^i\wedge dt
\]
where the $E_i(L)$ are the Euler-Lagrange expressions.  The 2-forms
which, like $\delta(Ldt)$, are linear combinations of the
$\theta^i\wedge dt$ are called source forms in \cite{vitolo} and
dynamical forms in \cite{olgageoff}.  Secondly, for any source form
$\varepsilon=f_i\theta^i\wedge dt$, $f_i=f_i(t,q,\dot{q},\ddot{q})$,
\[
\delta \varepsilon=-\onehalf(r_{ij}\theta^i\wedge\theta^j
+s_{ij}\theta^i\wedge\dot{\theta}^j
+t_{ij}\theta^i\wedge\ddot{\theta}^j)\wedge dt,
\]
where the coefficients are given by
\begin{align*}
r_{ij}&=\fpd{f_i}{q^j}-\fpd{f_j}{q^i}
-\onehalf\frac{d}{dt}\left(\fpd{f_i}{\dot{q}^j}-\fpd{f_j}{\dot{q}^i}\right)
+\onehalf\frac{d^2}{dt^2}\left(\fpd{f_i}{\ddot{q}^j}-\fpd{f_j}{\ddot{q}^i}\right)\\
s_{ij}&=\fpd{f_i}{\dot{q}^j}+\fpd{f_j}{\dot{q}^i}
-2\frac{d}{dt}\left(\fpd{f_j}{\ddot{q}^i}\right)\\
t_{ij}&=\fpd{f_i}{\ddot{q}^j}-\fpd{f_j}{\ddot{q}^i}.
\end{align*}
Again, $\delta^2=0$ and $\delta$ is locally exact.  With a source
form $\varepsilon=f_i\theta^i\wedge dt$,
$f_i=f_i(t,q,\dot{q},\ddot{q})$, one associates the second-order
system $f_i=0$, and conversely; so that $\delta \varepsilon=0$ is
necessary and sufficient for the second-order system $f_i=0$ to be
locally of Euler-Lagrange type.  The vanishing of the coefficients
$r_{ij}$, $s_{ij}$ and $t_{ij}$ are the (classical) Helmholtz
conditions (see e.g.\ \cite{germ1,olgageoff,Sant}).

By considering the transformation properties of the jet coordinates
and the contact forms under transformations of the form $\bar{q}^i=
\bar{q}^i(t,q)$, $\bar{t}=t$ one can show that the set of forms
spanned by $\{dt,\theta^i,\dot{\theta}^i\}$ with coefficients which
are functions of $t$, $q$ and $\dot{q}$ is well-defined.  We call
forms like this first-order forms.  Note that $d_V$ maps first-order
forms to first-order forms.  Moreover, one proves exactness of $d_V$
by using essentially the homotopy operator for the de Rham complex
for the variables $q$, $\dot{q}$, \ldots, treating $t$ as a
parameter.  It follows that if $\alpha$ is of first order and
satisfies $d_V\alpha=0$ then there is a {\em first-order\/} form
$\beta$ such that $\alpha=d_V\beta$.

The first step in applying these concepts to dissipative systems is
to characterize dissipative force terms using them.

\begin{prop}
The first-order source form
\[
\Delta=\fpd{D}{\dot{q}^i}\theta^i\wedge dt
\]
satisfies $\delta\Delta=d_V\Delta$.  Conversely, if $\varepsilon$ is
a first-order source form such that
$\delta\varepsilon=d_V\varepsilon$ then $\varepsilon=\Delta$ for
some first-order function $D$.
\end{prop}

\begin{proof}
Applying the formula for $\delta$ acting on a source form one finds that
\begin{align*}
\delta\Delta
&=
-\spd{D}{q^j}{\dot{q}^i}\theta^i\wedge\theta^j\wedge dt
-\spd{D}{\dot{q}^i}{\dot{q}^j}\theta^i\wedge\dot{\theta}^j\wedge dt\\
&=\left(\fpd{}{q^j}\left(\fpd{D}{\dot{q}^i}\right)\theta^j
+\fpd{}{\dot{q}^j}\left(\fpd{D}{\dot{q}^i}\right)\dot{\theta}^j\right)\wedge
\theta^i\wedge dt\\
&=d_V\Delta.
\end{align*}
Conversely, if $\varepsilon=f_i\theta^i\wedge dt$ is of
first order then $t_{ij}=0$ and
\[
s_{ij}=\fpd{f_i}{\dot{q}^j}+\fpd{f_j}{\dot{q}^i}.
\]
But
\[
d_V\varepsilon=\fpd{f_j}{q^i}\theta^i\wedge\theta^j\wedge dt-
\fpd{f_i}{\dot{q}^j}\theta^i\wedge\dot{\theta}^j\wedge dt,
\]
and so if $\delta\varepsilon=d_V\varepsilon$ then
\[
s_{ij}=\fpd{f_i}{\dot{q}^j}+\fpd{f_j}{\dot{q}^i}=2\fpd{f_i}{\dot{q}^j} \quad\mbox{or}\quad
\fpd{f_j}{\dot{q}^i}=\fpd{f_i}{\dot{q}^j},
\]
so that there is a function $D=D(t,q,\dot{q})$ such that $f_i=\partial
D/\partial\dot{q}^i$.  The terms in $\theta^i\wedge\theta^j\wedge dt$
then agree.
\end{proof}

Now take $\varepsilon$ to be the source form representing the given
equations. If they are of Euler-Lagrange type with a dissipative term
then $\varepsilon=\delta(Ldt)-\Delta$, so
$\delta\varepsilon=\delta\Delta=d_V\Delta$. Then
$\delta\varepsilon$ is of first order, and furthermore
$d_V\delta\varepsilon=0$. These are necessary conditions for the
given system to take the desired form. They are in fact sufficient
also, as we now show.

\begin{thm}
A system of second-order ordinary differential equations
$f_i(t,q,\dot{q},\ddot{q})=0$ may be written locally as
\[
\frac{d}{dt}\left(\fpd{L}{\dot{q}^i}\right)-\fpd{L}{q^i}=\fpd{D}{\dot{q}^i}
\]
for some first-order functions $L$ and $D$ if and only if the
corresponding source form $\varepsilon$ is such that
$\delta\varepsilon$ is of first order and satisfies
$d_V\delta\varepsilon=0$.
\end{thm}

\begin{proof}
It remains to prove sufficiency.  By the local exactness of $d_V$ we
may assume that $\delta\varepsilon=d_V\alpha$ for some first-order
2-form $\alpha$ (not necessarily a source form), which is determined
only up to the addition of a $d_V$-exact form. Contact 2-forms
$\beta$ can be ignored in $\alpha$ since $d_V\beta$ then contains no
$dt$ terms and therefore cannot contribute to $\delta\varepsilon$.
Moreover, if we put
\[
\alpha=(\lambda_i\theta^i+\mu_i\dot{\theta}^i)\wedge dt,
\]
since $\delta\varepsilon$ contains no
$\dot{\theta}^i\wedge\dot{\theta}^j$ terms either, we must have
\[
\fpd{\mu_j}{\dot{q}^i}=\fpd{\mu_i}{\dot{q}^j},
\]
so that $\mu_i=\partial\psi/\partial\dot{q}^i$ for some function
$\psi(t,q,\dot{q})$.  Hence, up to the $d_V$-exact term $d_V(\psi
dt)$, $\alpha$ is of the form
\[
\alpha=\left(\lambda_i-\fpd{\psi}{q^i}\right)\theta^i\wedge dt,
\]
i.e.\ is a source form, say $\alpha=\nu_i\theta^i\wedge dt$. Since
$\delta\varepsilon$ is of first order we must have $t_{ij}=0$,
whence $s_{ij}$ is symmetric in $i$ and $j$.  It follows that the
coefficient of $\theta^i\wedge\dot{\theta}^j$ in $d_V\alpha$ must be
symmetric in $i$ and $j$, whence $\nu_i=-\partial D/\partial
\dot{q}^i$ for some function $D=D(t,q,\dot{q})$, and
$\alpha=-\Delta$.  Then
$\delta\varepsilon=d_V\alpha=-d_V\Delta=-\delta\Delta$, and so there
is some $L$ such that $\varepsilon=\delta(Ldt)-\Delta$, as required.
\end{proof}

The point of interest that we mentioned at the beginning of this
appendix is that this version of the result represents the
conditions in part as the closure (under $d_V$) of a certain form,
namely $\delta\varepsilon$.  That it might be possible to state the
conditions in such a way was raised speculatively in \cite{germ1},
but the form and operator were not specifically identified there.

Under the assumption that $\delta\varepsilon$ is of first order, so
that $t_{ij}=0$, it is easy to verify that the $d_V$-closure
conditions are
\begin{align*}
0&=\fpd{r_{ij}}{q^k}+\fpd{r_{jk}}{q^i}+\fpd{r_{ki}}{q^j}\\
\fpd{r_{ij}}{\dot{q}^k}&=\fpd{s_{ik}}{q^j}-\fpd{s_{jk}}{q^i}\\
\fpd{s_{ij}}{\dot{q}^k}&=\fpd{s_{ik}}{\dot{q}^j}.
\end{align*}
It must not be forgotten that $r_{ij}$ and $s_{ij}$ are supposed to
be of first order; in addition, $t_{ij}=0$ means that we have
\[
\fpd{f_i}{\ddot{q}^j}=\fpd{f_j}{\ddot{q}^i}.
\]
These are the generalized Helmholtz conditions as given in
\cite{germ1}. However, it turns out that the first and last of the
closure conditions are consequences of the other conditions, as we
showed in \cite{zammnote}. It then follows easily that the following
conditions are equivalent to those given above:\
$f_i=g_{ij}\ddot{q}^j+h_i$ with $g_{ij}$ symmetric, where $g_{ij}$,
$h_i$ are of first order and further satisfy
\begin{align*}
&\fpd{g_{ij}}{\dot{q}^k}=\fpd{g_{ik}}{\dot{q}^k}\\
&\fpd{g_{ik}}{q^j}-\onehalf\spd{h_i}{\dot{q}^j}{\dot{q}^k}=
\fpd{g_{jk}}{q^i}-\onehalf\spd{h_j}{\dot{q}^i}{\dot{q}^k}\\
&\sum_{i,j,k}
\left(\spd{h_i}{q^j}{\dot{q}^k}-\spd{h_i}{q^k}{\dot{q}^j}\right)=0,
\end{align*}
where $\sum_{i,j,k}$ stands for the cyclic sum over the indices.

As the problem has been presented so far in this appendix, we must
take $g_{ij}$ and $h_i$ as given; the equations above provide a test
for determining whether the given second-order system can be put
into the required form.  However, it is now possible to regard these
equations from the alternative point of view:\ we set
$h_i=g_{ij}f^j$, where the second-order system is given in normal
form $\ddot{q}^i=f^i(t,q,\dot{q})$; we regard the $f^i$ as known but
the $g_{ij}$ as to be determined; the equations above now become
partial differential equations for the unknowns $g_{ij}$. We leave
it to the reader to verify that they correspond (take the autonomous
case for simplicity), in the order written, to the conditions of
Theorem~3, namely $\DV{}g$ is symmetric, $\DH{}g$ is symmetric, and
$\sum_{X,Y,Z}g(R(X,Y),Z)=0$.

\subsubsection*{Acknowledgements}
The first author is a Postdoctoral Fellow of the Research Foundation
-- Flanders (FWO). The second author is Adjunct Professor at La
Trobe University and acknowledges the hospitality of the Department
of Mathematics and Statistics during numerous visits. The third
author is a Guest Professor at Ghent University:\ he is grateful to
the Department of Mathematics for its hospitality.

\end{document}